\documentclass[11pt]{amsart}

\usepackage{amscd}
\usepackage{amsmath, amssymb}
\usepackage{amsthm}
\usepackage{amsfonts}
\usepackage{enumitem}
\usepackage{verbatim}
\newcommand{\ddt}{\frac{\partial}{\partial t}}
\newcommand{\ddbar}{\sqrt{-1} \partial \overline{\partial}}
\newcommand{\Ric}{\mathrm{Ric}}
\newcommand{\ov}[1]{\overline{#1}}
\newcommand{\tr}[2]{\mathrm{tr}_{#1}{#2}}
\newcommand{\omegahatt}{\hat{\omega}_t}

\newcommand{\omegatilde}{\tilde{\omega}}
\newcommand{\gtilde}{\tilde{g}}
\newcommand{\omegahat}{\hat{\omega}}
\newcommand{\ghat}{\hat{g}}
\newcommand{\C}{\mathbb{C}}
\newcommand{\R}{\mathbb{R}}

\title{Long time existence of the $(n-1)$-plurisubharmonic flow}
\author{Matthew Gill}
\thanks{Supported by NSF RTG grant DMS-0838703.}
\address{Department of Mathematics, University of California, Berkeley, 970 Evans Hall \#3840, Berkeley, CA 94720-3840 USA}
\email{mfgill@math.berkeley.edu	}

\begin{document}
\newcounter{remark}
\newcounter{theor}
\setcounter{remark}{0}
\setcounter{theor}{1}
\newtheorem{claim}{Claim}
\newtheorem{theorem}{Theorem}[section]
\newtheorem{lemma}[theorem]{Lemma}
\newtheorem{corollary}[theorem]{Corollary}
\newtheorem{proposition}[theorem]{Proposition}
\newtheorem{question}{question}[section]
\newtheorem{defn}{Definition}[theor]
\newtheorem{conjecture}[theorem]{Conjecture}

\begin{abstract}
We consider the $(n-1)$-plurisubharmonic flow, suggested by Tosatti-Weinkove, and prove a formula for its maximal time of existence. This includes estimates that will be useful in further investigating the flow.
\end{abstract}

\maketitle

\section{Introduction}

Let $M$ be a compact complex manifold of dimension $n > 2$ with $g$ and $g_0$ Hermitian metrics on $M$. We define the associated real $(1,1)$-form
\begin{equation*}
\omega = \sqrt{-1} g_{i\ov{j}} dz^i \wedge dz^{\ov{j}}
\end{equation*}
which will will also refer to as a metric. The $(n-1)$-plurisubharmonic flow is the equation
\begin{equation}\label{pshflow}
\ddt \omega_t^{n-1} = -(n-1) \Ric^C(\omega_t) \wedge \omega^{n-2}, \ \ \ \omega_t | _{t=0} = \omega_0.
\end{equation}
where $\Ric^C(\omega_t) = -\ddbar \log \omega_t^n$ is the Chern-Ricci form of $\omega_t$. In the case of $n=2$, \eqref{pshflow} becomes the Chern-Ricci flow (see \cite{G,G3,GS,LV,Nie,ShW2,TW2,TW3,TWY}). This flow was originally suggested by Tosatti-Weinkove in their work on the elliptic Monge-Amp\`ere equation for $(n-1)$-plurisubharmonic forms \cite{TW4,TW5}.

We say that a metric $\omega_0$ is {\em balanced} \cite{Michelsohn} if
\begin{equation*}
d \omega_0^{n-1} = 0
\end{equation*}
{\em Gauduchon} \cite{Gaud1} if
\begin{equation*}
\partial \ov{\partial} \omega^{n-1} = 0
\end{equation*}
and {\em strongly Gauduchon} (recently introduced by Popovici in \cite{Popo1}) if
\begin{equation*}
\ov{\partial} \omega_0^{n-1} \textrm{ is } \partial \textrm{-exact}.
\end{equation*}
When $\omega$ is a K\"ahler metric
\begin{equation*}
d\omega = 0
\end{equation*}
then the $(n-1)$-plurisubharmonic flow preserves all three of the above conditions imposed on $\omega_0$. If instead $\omega$ is an {\em Astheno-K\"ahler} metric (see \cite{JY})
\begin{equation*}
\partial \ov{\partial} \omega^{n-2} = 0
\end{equation*}
the flow preserves the Gauduchon and strongly Gauduchon conditions, but not necessarily the balanced condition. Indeed, the flow is equivalent to
\begin{equation*}
\ddt \omega_t^{n-1} = -(n-1) \Ric^C(\omega) \wedge \omega^{n-2} + \ddbar \theta(t) \wedge \omega^{n-2}
\end{equation*}
where
\begin{equation*}
\theta(t) = \log \frac{\det (g_t)^{n-1}}{\det g^{n-1}}.
\end{equation*}
Defining
\begin{equation*}
\Phi_t = \omega_0^{n-1} - t (n-1) \Ric^C(\omega) \wedge \omega^{n-2}
\end{equation*}
we see that a solution to \eqref{pshflow} is of the form
\begin{equation*}
\omega_t^{n-1} = \Phi_t + \ddbar u \wedge \omega^{n-2}  
\end{equation*} 
for some real valued function $u$ on $M$. One can check that if $\omega$ is K\"ahler and $\omega_0$ is balanced (respectively Gauduchon, strongly Gauduchon), then the family of metrics $\omega_t$ is balanced (respectively Gauduchon, strongly Gauduchon) for all $t$ along the flow. Similarly for $\omega$ Astheno-K\"ahler and $\omega_0$ Gauduchon or strongly Gauduchon. 

We prove the following formula for the maximal time of existence of the flow assuming $\omega_0$ and $\omega$ are Hermitian metrics.
\begin{theorem}\label{pshflowthm}
Let $M$ be a compact complex manifold of dimension $n \geq 3$ and let $\omega_0$ and $\omega$ be Hermitian metrics on $M$. Then there exists a unique solution of the $(n-1)$-plurisubharmonic flow \eqref{pshflow} on the maximal time interval $[0,T)$ where
\begin{equation*}
T = \sup \left \{ t > 0 \ | \ \exists \psi \in C^{\infty}(M) \textrm{ {\em such that} } \Phi_t + \ddbar \psi \wedge \omega^{n-2} > 0 \right \}. 
\end{equation*}
\end{theorem}
Note that if we define an equivalence relation of real $(n-1,n-1)$-forms by
\begin{equation*}
\Psi \sim \Psi' \iff \Psi = \Psi' + \ddbar \psi \wedge \omega^{n-2} \textrm{ for some } \psi \in C^{\infty}(M)
\end{equation*}
then $T$ depends only on $\omega$ and the equivalence class of $\omega_0^{n-1}$. This is analogous to the result of Tian-Zhang for the K\"ahler-Ricci flow \cite{TZ} and of Tosatti-Weinkove for the Chern-Ricci flow \cite{TW2}. Much like these related results, this theorem suggests that the $(n-1)$-plurisubharmonic flow is a natural object of study that reflects the geometry of the manifolds. 

Every Hermitian metric is conformal to a Gauduchon metric \cite{Gaud1} on a compact complex manifold. However if $\omega$ is only assumed to be Gauduchon then the $(n-1)$-plurisubharmonic flow \eqref{pshflow} does not preserve the Gauduchon condition of $\omega_0$. To alleviate this problem we consider the new flow
\begin{equation}\label{generalpsh}
\ddt \omega_t^{n-1} = -(n-1) \Ric^C(\omega_t) \wedge \omega^{n-2} - (n-1) \textrm{Re} \left( \sqrt{-1} \partial \left( \log \omega_t^n \right) \wedge \ov{\partial} (\omega^{n-2}) \right).
\end{equation}
If the fixed metric $\omega$ is Gauduchon and the initial metric $\omega_0$ is Gauduchon or strongly Gauduchon, so is the solution to \eqref{generalpsh} for as long as it exists. To see this, we compute as above. A solution to this new flow \eqref{generalpsh} is of the form
\begin{equation*}
\omega_t^{n-1} = \hat{\Phi}_t + \ddbar u \wedge \omega^{n-2} + \textrm{Re} \left( \sqrt{-1} \partial u \wedge \ov{\partial}( \omega^{n-2} )\right)
\end{equation*}
where
\begin{equation*}
\hat{\Phi}_t =  \omega_0^{n-1} - t(n-1)\left( \Ric^C(\omega) \wedge \omega^{n-2} + \textrm{Re}\left(  \sqrt{-1} \partial \left( \log \omega^n \right) \wedge \ov{\partial}( \omega^{n-2})\right) \right).
\end{equation*}
We conjecture that this flow has a similar theorem for its maximal existence time, but we are currently unable to prove the estimates that would give this result. 
\begin{conjecture}
Let $M$ be a compact complex manifold of dimension $n \geq 3$, $\omega$ a Gauduchon metric, and $\omega_0$ a Hermitian metric on $M$. Then there exists a unique solution of \eqref{generalpsh} on the maximal time interval $[0,T)$ where
\begin{equation*}
T = \sup \left \{ t > 0 \ \Bigg\vert 
\begin{array}{cc}
& \exists \  \psi \in C^{\infty}(M) \textrm{ {\em such that} } \\
& \hat{\Phi}_t + \ddbar \psi \wedge \omega^{n-2} \\
& \ \ + \textrm{{\em Re}} \left( \sqrt{-1} \partial \psi \wedge \ov{\partial}( \omega^{n-2} )\right)> 0 
\end{array}
\right \}. 
\end{equation*}
\end{conjecture}

The estimates required to prove the above conjecture are the same as those needed to prove Gauduchon's conjecture:
\begin{conjecture}(Gauduchon, 1977 \cite{Gaud2})
Let $M$ be a compact complex manifold and let $\psi$ be a closed real $(1,1)$-form on $M$ with $[\psi] = c^{BC}_1(M)$. Then there exists a Gauduchon metric $\omegatilde$ on $M$ with
\begin{equation*}
\Ric^C(\omegatilde) = \psi.
\end{equation*}
\end{conjecture}
This is a generalization of the famous Calabi-Yau theorem in K\"ahler geometry \cite{Yau}. Popovici \cite{Popo2} and Tosatti-Weinkove \cite{TW5} have both recently shown that proving Gauduchon's conjecture is equivalent to solving
\begin{equation}\label{gauduchoneq}
\det \left( \Phi_u \right) = e^{F+b} \det \left( \omega^{n-1} \right)
\end{equation}
with
\begin{equation*}
\Phi_u = \omega_0^{n-1} + \ddbar u \wedge \omega^{n-2} + \textrm{Re} \left( \sqrt{-1} \partial u \wedge \ov{\partial} (\omega^{n-2}) \right) > 0
\end{equation*}
with $\sup_M u = 0$ and $\omega$ Gauduchon. The missing ingredient for the solution is a second order estimate for $u$ solving \eqref{gauduchoneq}. Consider \eqref{gauduchoneq} where we remove the last term in the definition of $\Phi_u$: 
\begin{equation}\label{elliptic}
\det \left(\omega_0^{n-1} + \ddbar u \wedge \omega^{n-2} \right) = e^{F+b} \det \left( \omega^{n-1} \right)
\end{equation}
with
\begin{equation*}
\omega_0^{n-1} + \ddbar u \wedge \omega^{n-2} > 0, \ \ \ \sup_M u = 0
\end{equation*}
Fu-Wang-Wu \cite{FWW2} proved that \eqref{elliptic} has a smooth solution when $\omega$ is K\"ahler and has nonnegative orthogonal bisectional curvature and Tosatti-Weinkove have proven this result with no assumptions on $\omega$ other than being a Hermitian metric \cite{TW4,TW5}. The estimates of \cite{TW5} are crucial in the proof of the main theorem which we now summarize.

The general strategy is similar to that of the analogous results for the K\"ahler-Ricci flow \cite{TZ} (see also \cite{SWnotes})  and Chern-Ricci flow \cite{TW2}. Note that the flow \eqref{pshflow} cannot exist beyond $T$ as defined in the main theorem, so we assume that the flow has a maximal time of existence $S < T$. The $(n-1)$-plurisubharmonic flow is reduced to the parabolic scalar flow
\begin{equation}\label{ddtu}
\ddt u = \log \frac{ \left( \omegahatt + \frac{1}{n-1}((\Delta u)\omega - \ddbar u ) \right)^n}{\Omega}, \ \ \ u|_{t=0} = 0
\end{equation}
with 
\begin{equation*}
\omegahatt + \frac{1}{n-1}\left((\Delta u)\omega - \ddbar u \right)  > 0
\end{equation*}
on $[0,S)$. The maximum principle gives uniform bounds for $u$, $\dot{u}$, and the volume form $\omegatilde_t^n$ where
\begin{equation*}
\omegatilde_t =  \omegahatt + \frac{1}{n-1}\left((\Delta u)\omega - \ddbar u \right).
\end{equation*}
We then apply the maximum principle to obtain the estimate
\begin{equation*}
\tr{\omega}{\omegatilde_t} \leq C \left( \sup_{M \times [0,S)} \vert \nabla u \vert^2_g + 1 \right)
\end{equation*}
which is the parabolic version of the estimate from \cite{TW5} and the proof uses many similar elements. Following \cite{TW4}, we use a Liouville theorem and blow-up argument to uniformly bound $|\nabla u|_g^2$.  Applying the Evans-Kyrlov method (see \cite{GT,L} and \cite{G} in the complex setting for parabolic equations) gives the $C^{2+\alpha}(M,g)$ estimate and then from standard parabolic theory we produce higher order estimates. This allows us to extend the flow beyond the time $S$ contradicting the maximality of $S$.

\section{Reduction to Monge-Amp\`{e}re and notation}

We define the Christoffel symbols of the Hermitian metric $g$ in local holomorphic coordinates $(z^1, \ldots, z^n)$ by 
\begin{equation*}
\Gamma_{ij}^k = g^{k \ov l} \partial_{i} g_{j\ov{l}}
\end{equation*}
and the covariant derivative with respect to $g$ by
\begin{equation*}
\nabla_i a_l = \partial_i a_l - \Gamma_{il}^p a_p.
\end{equation*}
The torsion of $g$ is the tensor
\begin{equation*}
T_{ij}^k = \Gamma_{ij}^k - \Gamma_{ji}^k.
\end{equation*}
Note that if $g$ is a K\"ahler metric, then $T_{ij}^k = 0$. The Chern curvature of $g$ is
\begin{equation*}
{R_{k\ov l i}}^p = -\partial_{\ov l} \Gamma_{ki}^p
\end{equation*}
and it obeys the usual commutation identities for curvature. For example,
\begin{equation*}
[\nabla_i, \nabla_{\ov j}] a_l = -{R_{i\ov j l}}^pa_p, \ \ \ [\nabla_i, \nabla_{\ov j}] a_{\ov m} = {{R_{i\ov j}}^{\ov q}}_{\ov m} \ov{a_q}.
\end{equation*}
We will make use of the commutation formulas
\begin{equation*}
u_{i \ov j l} = u_{i l \ov j} - u_p {R_{l \ov j i}}^p, \ u_{p \ov j \ov m} = u_{p \ov m \ov j} - \ov{ T_{mj}^q}u_{p\ov q}, \ u_{i \ov q l} - T_{li}^p u_{p \ov q}
\end{equation*}
\begin{equation}\label{commutationrelation}
u_{i\ov{j}l\ov{m}} = u_{l\ov{m} i \ov{j}} + u_{p \ov{j}} {R_{l \ov{m} i}}^p - u_{p \ov{m}} {R_{i\ov{j} l}}^p - T_{li}^p u_{p \ov{m} \ov{j}} - \overline{T_{mj}^q} u_{l \ov{q} i} - T_{il}^p \overline{T_{mj}^q} u_{p\ov{q}}.
\end{equation}
The Chern-Ricci form $\Ric^C(\omega)$ is given by
\begin{equation*}
\Ric^C (\omega) = \sqrt{-1} R_{i\ov{j}} dz^i \wedge dz^{\ov{j}}
\end{equation*}
where 
\begin{equation*}
R_{i\ov j} = g^{k \ov l} R_{i \ov{j} k \ov{l}} = - \partial_i \partial_{\ov j} \log \det g.
\end{equation*}

A real $(n-1,n-1)$-form $\Psi$ is defined to be positive definite if for every nonzero $(1,0)$-form $\gamma$,
\begin{equation*}
\Psi \wedge \sqrt{-1} \gamma \wedge \overline{\gamma} \geq 0
\end{equation*}
with equality if and only if $\gamma = 0$. The determinant of a $\Psi$ is given by the determinant of the matrix $(\Psi_{i\ov j})$ where
\begin{align*}
\Psi & = (\sqrt{-1})^{n-1} (n-1)! \sum_{i,j} (\textrm{sgn}(i,j))\Psi_{i\ov j} dz^1 \wedge d {\ov z}^1 \wedge \ldots \wedge \hat{dz^i} \wedge dz^i \wedge \ldots \\
& \ \ \ \wedge dz^j \wedge \hat{d {\ov z}^j} \wedge \ldots \wedge dz^n \wedge d {\ov z}^n.
\end{align*}
Using this formula,
\begin{equation*}
\det \left(\omega^{n-1}\right) = \left( \det g \right)^{n-1}.
\end{equation*}

We say that a constant $C > 0$ is uniform if it only depends on the initial data for the $(n-1)$-plurisubharmonic flow. In our calculations a uniform constant $C$ may change from line to line.

Now we set up the proof of the main theorem. Suppose that $S$ is such that $0 < S < T$. Then there exists a smooth function $\psi$ such that
\begin{equation}\label{PsiS}
\Psi_S := \Phi_S + \ddbar \psi \wedge \omega^{n-2} >  0.
\end{equation}
We define $\Psi_t$ to be the straight line path from $\omega_0^{n-1}$ to $\Psi_S$ on $[0,S]$
\begin{align}\label{psitdef}
\Psi_t & = \frac{1}{S}\left( (S-t) \omega_0^{n-1} + t \left( \Phi_S + \ddbar \psi \wedge \omega^{n-2} \right)\right) \\ \nonumber
& = \omega_0^{n-1} + t \chi \wedge \omega^{n-2}
\end{align}
where $\chi = \frac{1}{S}\ddbar \psi - (n-1)\Ric^C (\omega).$ From its definition, note that $\Psi_t$ is uniformly bounded in the sense that there exists a uniform constant $C$ such that
\begin{equation}\label{psitbound}
\frac{1}{C} \omega^{n-1} \leq \Psi_t \leq C \omega^{n-1}
\end{equation} 
on $M \times [0,S]$. Define a family of Hermitian metrics $\omegahatt$ by
\begin{equation*}
\omegahatt = \frac{1}{(n-1)!} * \Psi_t = \hat{\omega}_0 + \frac{t}{n-1}\left( \left( \tr{\omega}{\chi} \right) \omega - \chi \right)
\end{equation*}
where $*$ is the Hodge star operator with respect to $g$ and
\begin{equation*}
\hat{\omega}_0 = \frac{1}{(n-1)!} * \omega_0^{n-1}.
\end{equation*}
From \eqref{psitbound} we also have
\begin{equation}\label{omegahatbound}
\frac{1}{C} \omega \leq \omegahatt \leq C \omega.
\end{equation}
on $M \times [0,S]$ for some uniform $C$. 

Suppose that $u$ satisfies \eqref{ddtu}
\begin{equation*}
\ddt u = \log \frac{ \left( \omegahatt + \frac{1}{n-1}((\Delta u)\omega - \ddbar u ) \right)^n}{\Omega}, \ \ \ u|_{t=0} = 0
\end{equation*}
with $\omegahatt + \frac{1}{n-1}((\Delta u)\omega - \ddbar u )  > 0$ and $\Omega := e^{\psi/S}\omega^n$. Note that
\begin{align}\label{ddtu2}
\ddt u & = \log \frac{ \left( \omegahatt + \frac{1}{n-1}((\Delta u)\omega - \ddbar u ) \right)^n}{\omega^n} -  \log e^{\psi/S} \\ \nonumber
& = \log \frac{\det *\left(\omegahatt + \frac{1}{n-1}((\Delta u)\omega - \ddbar u )\right)}{\det * \omega} - \frac{1}{S} \psi \\ \nonumber
& = \log \frac{\det \left( \Psi_t + \ddbar u \wedge \omega^{n-2} \right)}{\det \omega^{n-1}} - \frac{1}{S} \psi. 
\end{align}

Then if we define 
\begin{equation}\label{omegatdef}
\omega_t^{n-1} := \Psi_t + \ddbar u \wedge \omega^{n-2},
\end{equation}
equations \eqref{psitdef} and \eqref{ddtu2} show that
\begin{align}
\ddt \omega_t^{n-1} & = \chi \wedge \omega^{n-2} + \ddbar \ddt u \wedge \omega^{n-2} \\
&= -(n-1) \Ric^C(\omega_t) \wedge \omega^{n-2}. \nonumber
\end{align}

Conversely, suppose that $\omega_t^{n-1}$ as defined in \eqref{omegatdef} satisfies \eqref{pshflow}, then
\begin{align*}
\ddbar \left( \ddt u \right) \wedge \omega^{n-2} & = \ddt \left( \omega_t^{n-1} - \Psi_t \right) \\
& = \left( \ddbar \log \frac{\det \omega_t^{n-1}}{\omega^{n-1}} - \frac{1}{S}\ddbar \psi\right) \wedge \omega^{n-2}. 
\end{align*}
Using the equalities in \eqref{ddtu2}, we see that $\omega_t^{n-1}$ satisfies \eqref{pshflow} if and only if $u$ satisfies \eqref{ddtu}. 

We define the Hermitian metric $\omegatilde$ by 
\begin{equation}\label{omegatildedef}
\omegatilde_t :=  \omegahatt + \frac{1}{n-1}((\Delta u)\omega - \ddbar u  ).
\end{equation}
To simplify notation we drop the $t$ subscripts on the metrics and use $\omegatilde$ and $\omegahat$ to denote $\omegatilde_t$ and $\omegahatt$. However, $\omega$ will still denote the fixed Hermitian metric $\omega$ and we will not refer to the family of metrics $\omega_t$ solving \eqref{pshflow} for the remainder of this paper.

\section{Preliminary estimates}

We prove uniform bounds for $u$, $\dot{u}$, and the volume form $\omegatilde^n$. The estimate for $u$ is actually simpler than in the elliptic case \cite{TW4,TW5} since we can apply the parabolic maximum principle to \eqref{ddtu}.

\begin{lemma}\label{0est}
Suppose $u$ satisfies \eqref{ddtu} on $M \times [0,S)$. Then there exists a uniform $C > 0$ such that
\begin{enumerate}
\item $|u| \leq C $ 
\item $|\dot{u}| \leq C$ 
\item $\frac{1}{C} \Omega \leq \omegatilde^n \leq C \Omega$ 
\end{enumerate}
on $M \times [0,S)$.
\end{lemma}

To prove this, we need a maximum principle that will work in this context.
\begin{lemma}
Let $v$ be a smooth real-valued function on a compact complex manifold $M$ with Hermitian metric $\omega$. Then at a point $x_0$ where $v$ achieves a maximum, 
\begin{equation*}
\left( \Delta v \right) \omega - \ddbar v \leq 0.
\end{equation*} 
\end{lemma}
\begin{proof}
Choose coordinates at $x_0$ so that $g_{i \ov{j}} = \delta_{i \ov{j}}$ and $v_{i\ov{j}} := \partial_i \partial_{\ov{j}} v = \lambda_i \delta_{i\ov{j}}.$ Since $x_0$ is where $v$ attains a maximum $\lambda_i \leq 0$ for all $i = 0, \ldots, n$. Then at $x_0$,
\begin{equation*}
\left(\Delta v\right)g_{i\ov{j}} = \left(\sum_{i=1}^n \lambda_i \right) \delta_{i\ov{j}} \leq \lambda_i \delta_{i\ov{j}} = v_{i\ov{j}}.
\end{equation*} 
\end{proof}

We will also make use of the tensor
\begin{equation*}
\Theta^{i\ov{j}} = \frac{1}{n-1} \left( (\tr{\gtilde}{g}) g^{i\ov{j}} - \gtilde^{i\ov{j}} \right) > 0
\end{equation*}
and the operator $L$ acting on smooth functions $v$ on $M$ defined by
\begin{equation*}
Lv = \Theta^{i\ov{j}} \partial_i \partial_{\ov{j}} v.
\end{equation*}
Taking trace of \eqref{omegatildedef}, we have the useful relation
\begin{equation}\label{lu}
n = \tr{\omegatilde}{\omegahat} + Lu.
\end{equation}
Using this, we can prove Lemma \ref{0est} via maximum principle similar to the analogous estimates for the K\"ahler-Ricci flow (see \cite{SWnotes} for example).

\begin{proof}
For $(1)$, define a quantity $Q = u - At$ where $A$ is a constant to be determined later and fix $0 < t' < S$. Then suppose that a maximum of $Q$ on $M \times [0,t']$ occurs at a point $(x_0, t_0)$ with $t_0 > 0$. Applying the previous lemma and the usual maximum principle at $(x_0,t_0)$,
\begin{align*}
0 & \leq \ddt Q \\
& = \log \frac{ \left( \omegahat + \frac{1}{n-1}((\Delta u)\omega - \ddbar u ) \right)^n}{\Omega} - A \\
& \leq \log \frac{\omegahat^n}{\Omega} - A \\
& \leq C - A
\end{align*}
where on the last line we used \eqref{omegahatbound}. Choosing $A = C+1$, we get a contradiction. Since $t'$ is arbitrary, we conclude that $Q$ achieves its maximum at $t_0 = 0$ and so we have a uniform upper bound for $u$. The lower bound follows similarly.

For $(2)$, we compute the evolution equation for $\dot{u}$. Using \eqref{ddtu},
\begin{equation}\label{udotdot}
\ddt \dot{u} = \tr{\omegatilde}{\left( \ddt \omegatilde \right)} = \frac{1}{n-1}\tr{\omegatilde}{\left( (\tr{\omega}{\chi}) \omega - \chi + (\Delta \dot{u}) \omega - \ddbar \dot{u}  \right)}
\end{equation}
Then we have
\begin{align}\label{ludot}
L\dot{u} &= \frac{1}{n-1} \left( (\tr{\gtilde}{g}) g^{i\ov{j}} - \gtilde^{i\ov{j}} \right) \partial_i \partial_{\ov{j}} \dot{u} \\ \nonumber
&= \frac{1}{n-1} \left( (\Delta \dot{u})(\tr{\omegatilde}{\omega}) - \tr{\omegatilde}{\ddbar \dot{u}}\right).
\end{align}
Now consider the quantity $Q = (n-1)\dot{u} - Au$ where $A$ is a constant to be determined. Combining \eqref{lu}, \eqref{udotdot}, and \eqref{ludot},
\begin{equation*}
\left( \ddt - L \right) Q = (\tr{\omegatilde}{\omega})(\tr{\omega}{\chi}) - \tr{\omegatilde}{\chi} - A \dot{u} + An - A \tr{\omegatilde}{\omegahat}.
\end{equation*}
Using \ref{omegahatbound}, we can choose $A$ large enough so that
\begin{equation*}
A \omegahat \geq (\tr{\omega}{\chi}) \omega - \chi
\end{equation*}
which gives
\begin{equation*}
\left( \ddt - L \right) Q \leq -A \dot{u} + An.
\end{equation*}
Hence at a point $(x_0, t_0)$ at which $Q$ achieves a maximum, $\dot{u}(x_0,t_0) \leq n$. Then since $Q$ is bounded above by its value at $(x_0,t_0)$,
\begin{equation*}
\dot{u} \leq \frac{1}{n-1}\left( A\sup_{M\times [0,S)} u + n(n-1) -Au(x_0,t_0) \right) \leq C
\end{equation*}
where for the last inequality we used the above uniform bound for $u$. 

To prove the lower bound, consider the quantity 
\begin{equation*}
Q = (n-1)(S - t + \varepsilon)\dot{u} + u + nt
\end{equation*} 
where $\epsilon > 0$ is a constant to be determined. Again applying \eqref{lu}, \eqref{udotdot}, and \eqref{ludot},
\begin{align*}
\left(\ddt - L \right) Q &= - \dot{u} + (S - t + \varepsilon)\big( (\tr{\omegatilde}{\omega})(\tr{\omega}{\chi}) - \tr{\omegatilde}{\chi} \big) + \dot{u} - n + \tr{\omegatilde}{\omegahat} + n \\
&=\tr{\omegatilde}{\big( \omegahat_S + \varepsilon ( (\tr{\omega}{\chi}) \omega - \chi) \big)} \\
&> 0
\end{align*} 
provided we choose $\varepsilon > 0$ small enough. If $Q$ achieves a minimum at a point $(x_0,t_0)$ with $t_0 > 0$, we have a contradiction. Hence $Q$ must be bounded from below by its infimum over $M$ at time $t = 0$. When combined with the uniform bound for $u$, this gives the lower bound for $\dot{u}$.

To finish the lemma, (3) follows immediately from (2) since we have 
\begin{equation*}
\dot{u} = \log \frac{\omegatilde^n}{\Omega}.
\end{equation*}
\end{proof}

\section{Second order estimate}

We obtain a second order estimate for $u$ in terms of $\tr{\omega}{\omegatilde}$. This estimate is the parabolic version of the estimates from  Hou-Ma-Wu \cite{HMW} and Tosatti-Weinkove \cite{TW4,TW5} and the proof follows a similar method.
\begin{lemma}\label{2est}
There exists a uniform $C > 0$ such that
\begin{equation}\label{tracebound}
\tr{\omega}{\omegatilde} \leq C \left( \sup_{M \times [0,S)} \vert \nabla u \vert^2_g + 1 \right)
\end{equation}
on $M \times [0,S)$.
\end{lemma}

\begin{proof}
As in \cite{TW5} we consider the tensor
\begin{equation}\label{etaij}
\eta_{i\ov{j}} = u_{i\ov{j}} + (\tr{g}{\ghat}) g_{i\ov{j}} - (n-1) \ghat_{i\ov{j}} = (\tr{g}{\gtilde}) g_{i\ov{j}} - (n-1)\gtilde_{i\ov{j}}.
\end{equation}
Fix a $t'$ such that $0 < t' < S$. Define the quantity
\begin{equation*}
H(x,\xi,t) = \log( \eta_{i\ov{j}} \xi^i \overline{\xi^{j}}) + c \log \left( g^{p\ov{q}} \eta_{i\ov{q}} \eta_{p\ov{j}} \xi^i \overline{\xi^j} \right) + \varphi \left( | \nabla u |_g^2 \right) + \nu(u)
\end{equation*}
for $x \in M, \xi \in T^{1,0}_x M$ a $g$-unit vector, $t \in [0,t']$, and $c > 0$ a small constant to be determined.  The above functions are
\begin{align*}
\varphi(s) &= -\frac{1}{2} \log \left( 1 - \frac{s}{2K} \right), \ \ \ 0 \leq s \leq K - 1 \\
\nu(s) &= -A \log \left( 1 + \frac{s}{2L}\right), \ \ \ -L + 1 \leq s \leq L - 1,
\end{align*} 
where
\begin{equation*}
K = \sup_{M \times [0,t']} | \nabla u |_g^2 + 1, L = \sup_{M \times [0,S)} |u| + 1, A = 3L(C_1 + 1)
\end{equation*}
with $C_1$ a uniform constant to be determined during the proof. Note that $L$ is uniformly bounded by Lemma \ref{0est}. This setup is similar to \cite{HMW, TW4, TW5}, the difference being that we have a time dependence. Evaluating at $|\nabla u|_2^2$, we have the bounds
\begin{equation}\label{phibounds}
0 \leq \varphi \leq C, \ \ \ 0 < \frac{1}{4K} \leq \varphi ' \leq \frac{1}{2K}, \ \ \ \varphi '' = 2 ( \varphi ' )^2 > 0
\end{equation}
and evaluating at $u$,
\begin{equation}\label{nubounds}
| \nu | \leq C, \ \ \ C_1 + 1 = \frac{A}{3L} \leq - \nu ' \leq \frac{A}{L}, \ \ \ \frac{2\varepsilon}{1-\varepsilon} (\nu ' )^2 \leq \nu '', \textrm{ for all } \varepsilon \leq \frac{1}{2A+1}
\end{equation}
on $M \times [0,t']$ for uniform $C > 0$.

Similar to \cite{HMW}, we define the set
\begin{equation*}
W = \left \{ (x,\xi,t) \ \big| \ \eta(x,t)_{i\ov{j}} \xi^i \overline{\xi^{j}} \geq 0, \xi \in T^{1,0}_x M \textrm{ a }g\textrm{-unit vector}, t \in [0,t'] \right \}.  
\end{equation*}
Then $W$ is compact, $H = -\infty$ on the boundary of a cross section $W_t$ for fixed time $t$, and $H$ is upper semi-continuous on $W_t$. Thus if $H$ has a maximum at a point $(x_0, \xi_0, t_0)$ in $W$,  $(x_0,\xi_0)$ is in the interior of $W_{t_0}$. We assume without loss of generality that $t_0 > 0$. 

Choose holomorphic coordinates $(z^1, \ldots, z^n)$ centered at $x_0$ such that at $(x_0,t_0)$
\begin{equation*}
g_{i\ov{j}} = \delta_{i\ov{j}}, \ \ \ \eta_{i\ov{j}} = \delta_{i\ov{j}} \eta_{i\ov{i}}, \ \ \ \eta_{1\ov{1}} \geq \eta_{2 \ov{2}} \geq \ldots \geq \eta_{n \ov{n}}.
\end{equation*}
From the definition of $\eta_{i\ov{j}}$
\begin{equation*}
\gtilde_{i\ov{j}} = \frac{1}{n-1} \left ( -\eta_{i\ov{j}} + (\tr{g}{\gtilde})g_{i\ov{j}} \right )
\end{equation*}
so that $\gtilde_{i\ov{j}}$ is also diagonal at $(x_0,t_0)$ and we may define $\lambda_i$ by
\begin{equation*}
\gtilde_{i\ov{j}} = \lambda_i \delta_{i\ov{j}}. 
\end{equation*}
at $(x_0,t_0)$. Using \eqref{etaij},
\begin{equation}\label{eiganvalsum}
\eta_{i\ov{i}} = \sum_{j = 1}^n \lambda_j - (n-1) \lambda_i
\end{equation}
which gives 
\begin{equation*}
0 < \lambda_1 \leq \ldots \leq \lambda_n
\end{equation*}
and
\begin{equation}\label{bigunifequiv}
\frac{1}{n} \tr{\omega}{\omegatilde} \leq \lambda_n \leq \eta_{1\ov{1}} \leq (n-1) \lambda_n \leq (n-1) \tr{\omega}{\omegatilde}.
\end{equation}

Following \cite{TW5}, choosing $c < 1/(n-3)$ when $n > 3$ or $c$ any positive real number when $n = 3$, the quantity 
\begin{equation*}
 \log( \eta_{i\ov{j}} \xi^i \overline{\xi^{j}}) + c \log \left( g^{p\ov{q}} \eta_{i\ov{q}} \eta_{p\ov{j}} \xi^i \overline{\xi^j} \right) 
\end{equation*}
is maximized at $(x_0,t_0)$ by $\xi_0 = \partial / \partial z^1$ since $\eta_{1\ov{1}}$ is the largest eigenvalue of $\eta_{i\ov{j}}$. We extend $\xi_0$ over our coordinate patch to the unit vector field 
\begin{equation*}
\xi_0 = g^{-1/2}_{1\ov{1}} \frac{\partial}{\partial z^1}.
\end{equation*}
Now we consider the quantity
\begin{equation}\label{secondQ}
Q(x,t) = H(x,\xi_0,t) = \log \left( g^{-1}_{1\ov{1}} \eta_{1\ov{1}} \right) + \varphi \left( | \nabla u |_g^2 \right) + \nu(u)
\end{equation}
defined in a neighborhood of $(x_0,t_0)$ chosen small enough so that $Q$ attains its maximum at $(x_0,t_0)$. The proof of the estimate follows from applying the maximum principle to this quantity to obtain the bound
\begin{equation}\label{etaleqK}
\eta_{1\ov{1}}(x_0,t_0) \leq C K = C\left( \sup_{M \times [0,t']} |\nabla u|_g^2 + 1 \right).
\end{equation}
which will complete the proof: at any point $(x,t) \in M \times [0,t']$ using \eqref{bigunifequiv},
\begin{align}\label{cpctinterval2est}
\tr{\omega}{\omegatilde}(x,t) & \leq n \eta_{1\ov 1}(x,t) \\  \nonumber
& \leq n \sup_W \left( (\eta_{i\ov j} \xi^i \ov{\xi^j})^{1/(1+2c)}\left( g^{p\ov q}  \eta_{i\ov{q}} \eta_{p\ov{j}} \xi^i \overline{\xi^j} \right)^{c/(1+2c)} \right) \\ \nonumber
& \leq Ce^{Q(x_0,t_0)} \\ \nonumber
& \leq C \left( \sup_{M \times [0,t']} |\nabla u|_g^2 + 1 \right).
\end{align}
Since $C > 0$ is uniform we get the desired estimate \eqref{tracebound}.

We begin the proof of the estimate \eqref{etaleqK}. First, we collect some useful facts. At the point $(x_0,t_0)$,
\begin{equation}\label{sumThetaii}
\sum_{i} \Theta^{i\ov{i}} = \tr{\gtilde}{g}
\end{equation}
and we may assume that at this point
\begin{equation}\label{uijbound2eta}
|u_{i\ov{j}}| \leq 2 |\eta_{1\ov{1}} |
\end{equation}
since our goal is to prove a uniform bound for $\eta_{1\ov 1}(x_0,t_0)$. As in \cite{TW5} we have at $(x_0,t_0)$
\begin{align}\label{LofQbound}
L(Q) \geq & (1+2c) \sum_i \frac{\Theta^{i\ov{i}} \eta_{1\ov{1} i \ov{i}}}{\eta_{1\ov{1}}} + \frac{c}{2} \sum_i \sum_{p \neq 1} \frac{\Theta^{i\ov{i}} |\eta_{p \ov{1} i} |^2}{(\eta_{1\ov{1}})^2} + \frac{c}{2} \sum_i \sum_{p \neq 1} \frac{\Theta^{i\ov{i}} |\eta_{1 \ov{p} i} |^2}{(\eta_{1\ov{1}})^2} \\ \nonumber
& \ - (1 + 2c) \sum_i \frac{ \Theta^{i\ov{i}} | \eta_{1 \ov{1} i } |^2 }{(\eta_{1\ov{1}})^2} + \nu ' \sum_i \Theta^{i\ov{i}} u_{i\ov{i}} + \nu '' \sum_i \Theta^{i\ov{i}} | u_i |^2 \\ \nonumber
& \ + \varphi '' \sum_{i} \Theta^{i\ov{i}} \left \vert \sum_{p} u_p u_{\ov{p}i} + \sum_p u_{p i} u_{\ov{p}} \right \vert^2 + \varphi ' \sum_{i,p} \Theta^{i\ov{i}} \left( | u_{p \ov{i}} |^2 + | u_{p i} |^2 \right) \\ \nonumber
& \ + \varphi ' \sum_{i,p} \Theta^{i\ov{i}} \left( u_{p i \ov{i}} u_{\ov{p}} + u_{\ov{p} i \ov{i}} u_p \right) - C \tr{\gtilde}{g}
\end{align}
for a uniform $C > 0$ where the subscripts denote covariant derivatives with respect to the fixed Hermitian metric $g$.

Computing the time evolution of $Q$ at $(x_0,t_0)$,
\begin{equation}\label{timeEvoQ}
\ddt Q = (1+2c) \frac{ \dot{\eta}_{1\ov{1}}}{\eta_{1\ov{1}}} + \varphi' \left( \sum_p \dot{u}_p u_{\ov{p}} + \sum_{p} \dot{u}_{\ov{p}} u_p \right) + \nu ' \dot{u}.
\end{equation}
Using the definition of $\eta_{i\ov{j}}$ \eqref{etaij},
\begin{align*}
\dot{\eta}_{i\ov{j}} &= \dot{u}_{i\ov{j}} + \left( \tr{g}{\ddt \ghat} \right) g_{i\ov{j}} - (n-1) \ddt \ghat_{i\ov{j}} \\ \nonumber
&= \dot{u}_{i\ov{j}} + (\tr{g}{\chi})g_{i \ov j} - \left( (\tr{g}{\chi}) g_{i\ov{j}} - \chi_{i\ov{j}} \right).
\end{align*}
Evaluating at $(x_0,t_0)$,
\begin{equation}\label{etadot11}
\dot{\eta}_{1\ov{1}} = \dot{u}_{1\ov{1}} + \chi_{1\ov{1}}.
\end{equation}
Covariantly differentiating the flow \eqref{ddtu} with respect to $g$,
\begin{equation}\label{covdiffudot}
\dot{u}_l =  \gtilde^{i\ov{j}} \nabla_l \gtilde_{i\ov{j}} - \frac{1}{S} \psi_l 
\end{equation}
and
\begin{equation}\label{covdiffudot2}
\dot{u}_{l \ov{m}} = \gtilde^{i\ov{j}} \nabla_{\ov{m}} \nabla_{l} \gtilde_{i\ov{j}} - \gtilde^{i \ov{q}} \gtilde^{p \ov{j}} \nabla_{\ov{m}} \gtilde_{p \ov{q}} \nabla_l \gtilde_{i\ov{j}} - \frac{1}{S} \psi_{l \ov{m}}.
\end{equation} 
Using the definition of $\gtilde$ \eqref{omegatildedef},
\begin{equation*}
\dot{u}_l =  \Theta^{i\ov{j}} u_{i \ov{j} l} + \gtilde^{i\ov{j}} \nabla_l \ghat_{i\ov{j}} - \frac{1}{S} \psi_l
\end{equation*}
and letting $\hat{h}_{i\ov{j}} = (n-1) \ghat_{i\ov{j}}$,
\begin{align*}
\dot{u}_{l\ov{m}} = & \Theta^{i\ov{j}} u_{i\ov{j} l \ov{m}} + \gtilde^{i\ov{j}} \nabla_{\ov{m}} \nabla_l \ghat_{i\ov{j}} - \frac{1}{S} \psi_{l\ov{m}} \\
& \ - \frac{\gtilde^{i\ov{q}} \gtilde^{p\ov{j}} \left( g_{p\ov{q}} g^{r\ov{s}} u_{r \ov{s} \ov{m}} - u_{p \ov{q} \ov{m}} + \nabla_{\ov{m}} \hat{h}_{p\ov{q}}\right) \left( g_{i\ov{j}} g^{r\ov{s}} u_{r \ov{s} l} - u_{i \ov{j} l} + \nabla_l \hat{h}_{i\ov{j}}\right)}{(n-1)^2}.
\end{align*}
At $(x_0,t_0)$, these become
\begin{equation}\label{dotup}
\dot{u}_p = \sum_i \Theta^{i\ov{i}} u_{i \ov{i} p} + \sum_{i} \gtilde^{i\ov{i}} \ghat_{i\ov{i}p} - \frac{1}{S} \psi_p
\end{equation}
and
\begin{align}\label{dotu11}
\dot{u}_{1\ov{1}} = & \sum_{i} \Theta^{i\ov{i}} u_{i\ov{i} 1 \ov{1}} + \sum_i \gtilde^{i\ov{i}} \ghat_{i\ov{i} 1 \ov{1}} - \frac{1}{S} \psi_{1\ov{1}} - H
\end{align}
where
\begin{equation}\label{Hdef}
H = \frac{\sum_{i,j} \gtilde^{i\ov{i}} \gtilde^{j\ov{j}} \left( g_{j\ov{i}} \sum_a u_{a\ov{a}\ov{i}} - u_{j \ov{i} \ov{1}} + \hat{h}_{j \ov{i} \ov{1}} \right) \left( g_{i\ov{j}} \sum_b u_{b \ov{b} 1} - u_{i\ov{j} 1} + \hat{h}_{i \ov{j} 1} \right)}{(n-1)^2}.
\end{equation}
Applying the commutation rule \eqref{commutationrelation}, \eqref{dotu11} becomes
\begin{align}\label{dotu112}
\dot{u}_{1\ov{1}} = & -H + \sum_{i} \Theta^{i\ov{i}} u_{ 1 \ov{1} i \ov{i}} + \sum_i \gtilde^{i\ov{i}} \ghat_{i\ov{i} 1 \ov{1}} - \frac{1}{S} \psi_{1\ov{1}} \\ \nonumber
& \ + \sum_i \Theta^{i\ov{i}} \left( u_{p \ov{i}} {R_{1\ov{1}i}}^p - u_{p\ov{1}} {R_{i\ov{i}1}}^p \right) \\ \nonumber
& \ - \sum_i \Theta^{i\ov{i}} \left( T_{1i}^p u_{p \ov{1} i} + \overline{T_{1i}^p} u_{1 \ov{p} i} + T_{i1}^p \overline{T_{1i}^q} u_{p\ov{q}} \right).
\end{align}
Combining \eqref{timeEvoQ}, \eqref{etadot11}, \eqref{dotu112}, and the fact that
\begin{equation*}
u_{1\ov{1}i\ov{i}} = \eta_{1\ov{1} i\ov{i}} + \hat{h}_{1\ov{1} i\ov{i}} - (\tr{g}{\ghat})_{i\ov{i}}
\end{equation*}
we have the evolution equation
\begin{align}\label{timeEvoQ2}
\ddt Q & = -(1+2c) \frac{H}{\eta_{1\ov 1}} + (1+2c) \sum_{i} \frac{\Theta^{i \ov i} \eta_{1 \ov 1 i \ov i}}{\eta_{1 \ov 1}} \\ \nonumber
& \ \ \ +  \frac{1+2c}{\eta_{1\ov{1}}} \Big( \chi_{1\ov{1}} - \frac{1}{S}\psi_{1\ov{1}} +  \sum_i \Theta^{i\ov{i}} \left( u_{p \ov{i}} {R_{1\ov{1}i}}^p - u_{p\ov{1}} {R_{i\ov{i}1}}^p \right) \\ \nonumber
& \ \ \  + \sum_i \gtilde^{i\ov{i}} \ghat_{i\ov{i} 1 \ov{1}} + \sum_{i} \Theta^{i\ov{i}} \left( \hat{h}_{1\ov{1}i\ov{i}} - (\tr{g}{\ghat})_{i\ov{i}} \right) \Big) \\ \nonumber
& \ \ \ - \frac{2(1+2c)}{\eta_{1\ov{1}}}  \sum_{i,p} \Theta^{i\ov{i}} \textrm{Re} \left( \overline{T_{1i}^p} u_{p \ov{1} i} \right) - \frac{1+2c}{\eta_{1\ov{1}}}  \sum_{i,p} \Theta^{i\ov{i}} T_{i1}^p \overline{T_{1i}^q} u_{p\ov{q}} \\ \nonumber
& \ \ \ + \varphi' \left( \sum_p \dot{u}_p u_{\ov{p}} + \sum_{p} \dot{u}_{\ov{p}} u_p \right) + \nu ' \dot{u}.
\end{align}
Subtracting \eqref{timeEvoQ2} and \eqref{LofQbound} we obtain the evolution equation bound at $(x_0,t_0)$,
\begin{align}\label{ddtLQ}
0 & \leq  \left( \ddt - L \right) Q \\ \nonumber
& \leq  -(1+2c)\frac{H}{\eta_{1 \ov 1}} \\ \nonumber
& \ \ \ - \frac{c}{2} \sum_i \sum_{p \neq 1} \frac{\Theta^{i\ov{i}} |\eta_{p \ov{1} i} |^2}{(\eta_{1\ov{1}})^2} - \frac{c}{2} \sum_i \sum_{p \neq 1} \frac{\Theta^{i\ov{i}} |\eta_{1 \ov{p} i} |^2}{(\eta_{1\ov{1}})^2} + (1 + 2c) \sum_i \frac{ \Theta^{i\ov{i}} | \eta_{1 \ov{1} i } |^2 }{(\eta_{1\ov{1}})^2} \\ \nonumber
& \ \ \ + C \tr{\gtilde}{g} + \frac{1+2c}{\eta_{1\ov{1}}} \Big( \chi_{1\ov{1}} - \frac{1}{S}\psi_{1\ov{1}} +  \sum_i \Theta^{i\ov{i}} \left( u_{p \ov{i}} {R_{1\ov{1}i}}^p - u_{p\ov{1}} {R_{i\ov{i}1}}^p \right) \\ \nonumber
& \ \ \ + \sum_i \gtilde^{i\ov{i}} \ghat_{i\ov{i} 1 \ov{1}} + \sum_{i} \Theta^{i\ov{i}} \left( \hat{h}_{1\ov{1}i\ov{i}} - (\tr{g}{\ghat})_{i\ov{i}} \right) \Big) \\ \nonumber
& \ \ \ - \frac{2(1+2c)}{\eta_{1\ov{1}}}  \sum_{i,p} \Theta^{i\ov{i}} \textrm{Re} \left( \overline{T_{1i}^p} u_{p \ov{1} i} \right) - \frac{1+2c}{\eta_{1\ov{1}}}  \sum_{i,p} \Theta^{i\ov{i}} T_{i1}^p \overline{T_{1i}^q} u_{p\ov{q}} \\ \nonumber
& \ \ \ + \nu ' \left(\ddt - L\right)u \\ \nonumber
& \ \ \ - \nu '' \sum_i \Theta^{i\ov{i}} |u_i|^2 -  \varphi '' \sum_{i} \Theta^{i\ov{i}} \left \vert \sum_{p} u_p u_{\ov{p}i} + \sum_p u_{p i} u_{\ov{p}} \right \vert^2 \\ \nonumber
& \ \ \ - \varphi ' \sum_{i,p} \Theta^{i\ov{i}} \left( | u_{p \ov{i}} |^2 + | u_{p i} |^2 \right)  \\ \nonumber 
& \ \ \ + \varphi ' \sum_p \left( \left( \dot{u}_p - \sum_i \Theta^{i\ov{i}} u_{p i \ov{i}} \right)  u_{\ov{p}} + u_{\ov{p}} \left( \dot{u}_{\ov{p}} - \sum_i \Theta^{i\ov{i}} u_{\ov{p} i \ov{i}} \right) \right) \\ \nonumber
& = (1) + (2) + (3) + (4) + (5) + (6) + (7) + (8) + (9)
\end{align}
where $(1)$ through $(9)$ correspond to the lines in the last inequality. We now bound each of the lines of \eqref{ddtLQ} from above.

\vspace{0.2in}

\noindent
Lines (3) and (4): Using \eqref{sumThetaii} and \eqref{uijbound2eta} we have the upper bound
\begin{equation*}
(3) + (4) \leq C \tr{\gtilde}{g} + C.
\end{equation*}

\vspace{0.2in}

\noindent
Line (5):  As in \cite{TW5}, using the second term from line (2) we can bound line (5). Covariantly differentiating \eqref{etaij},
\begin{equation*}
u_{1 \ov{p} i} = \eta_{1 \ov{p} i} - (\tr{g}{\ghat})_i g_{1 \ov{p}} + \hat{h}_{1 \ov{p} i}
\end{equation*}
and so
\begin{equation}\label{torsion1}
- \frac{2(1+2c)}{\eta_{1\ov{1}}}  \sum_{i,p} \Theta^{i\ov{i}} \textrm{Re} \left( \overline{T_{1i}^p} u_{p \ov{1} i} \right) \leq - \frac{2(1+2c)}{\eta_{1\ov{1}}}  \sum_{i,p} \Theta^{i\ov{i}} \textrm{Re} \left( \overline{T_{1i}^p} \eta_{p \ov{1} i} \right) + C \tr{\gtilde}{g}. 
\end{equation}
Since $T^1_{11} = 0$, the term from the sum with $p = 1$ is
\begin{equation}\label{torsion2}
-\frac{2(1+2c)}{\eta_{1\ov{1}}} \sum_{i} \Theta^{i\ov{i}} \textrm{Re} \left( \overline{T_{1i}^1} \eta_{1 \ov{1} i} \right) = -\frac{2(1+2c)}{\eta_{1\ov{1}}} \sum_{i \neq 1} \Theta^{i\ov{i}} \textrm{Re} \left( \overline{T_{1i}^1} \eta_{1 \ov{1} i} \right).
\end{equation}
The remaining summands can be bounded by
\begin{equation}\label{torsion3}
- \frac{2(1+2c)}{\eta_{1\ov{1}}}  \sum_{i} \sum_{p \neq 1} \Theta^{i\ov{i}} \textrm{Re} \left( \overline{T_{1i}^p} \eta_{p \ov{1} i} \right) \leq \frac{c}{4} \sum_i \sum_{p \neq 1} \Theta^{i\ov{i}} \frac{|\eta_{1 \ov{p} i}|^2}{(\eta_{1\ov{1}})^2} + C \tr{\gtilde}{g}.
\end{equation}
Putting together \eqref{torsion1}, \eqref{torsion2}, \eqref{torsion3} and controlling the second term in (6) using \eqref{uijbound2eta} we have the bound
\begin{equation*}
(6) \leq -\frac{2(1+2c)}{\eta_{1\ov{1}}} \sum_{i \neq 1} \Theta^{i\ov{i}} \textrm{Re} \left( \overline{T_{1i}^1} \eta_{1 \ov{1} i} \right) + \frac{c}{4} \sum_i \sum_{p \neq 1} \Theta^{i\ov{i}} \frac{|\eta_{1 \ov{p} i}|^2}{(\eta_{1\ov{1}})^2} + C \tr{\gtilde}{g}.
\end{equation*}

\vspace{0.2in}

\noindent
Line (6): Applying \eqref{lu}, the uniform bound for $\dot{u}$, and \eqref{nubounds},
\begin{align*}
(6) & =   \nu ' \dot{u} - n \nu ' + \nu ' \tr{\gtilde}{\ghat} \\ \nonumber
& \leq  3C (C_1+1) + 3(C_1 +1)n - (C_1 + 1) \tr{\gtilde}{\ghat} \\ \nonumber
& \leq  C - (C_1 + 1) \tr{\gtilde}{\ghat}
\end{align*}
remembering that $C_1 > 0$ is to be determined.

\vspace{0.2in}

\noindent
Lines (8) and (9): For line (9), commuting covariant derivatives and recalling \eqref{dotup}
\begin{align*}
(9) & =  \varphi ' \sum_p \left( \left( \dot{u}_p - \sum_i \Theta^{i\ov{i}} u_{ i \ov{i} p } \right)  u_{\ov{p}} + u_{\ov{p}} \left( \dot{u}_{\ov{p}} - \sum_i \Theta^{i\ov{i}} u_{ i \ov{i} \ov{p} } \right) \right) \\ \nonumber
& \ \ \ - \varphi ' \sum_{i,p} \Theta^{i\ov{i}} u_q u_{\ov{p}} {R_{i \ov{i} p}}^q + 2 \mathrm{Re} \varphi ' \sum_{i,p,q} \Theta^{i\ov{i}} u_{\ov{p}} u_{q\ov{i}} T_{ip}^q \\ \nonumber
& =  \varphi ' \sum_{i,p} \gtilde^{i\ov{i}} \left( \ghat_{i\ov{i}p} u_{\ov{p}} + \ghat_{i\ov{i}\ov{p}} u_{p}\right) - \varphi ' \sum_p \left( \frac{\psi_p}{S} u_{\ov{p}} + \frac{\psi_{\ov{p}}}{S} u_p \right) \\ \nonumber
& \ \ \ - \varphi ' \sum_{i,p} \Theta^{i\ov{i}} u_q u_{\ov{p}} {R_{i \ov{i} p}}^q + 2 \mathrm{Re} \varphi ' \sum_{i,p,q} \Theta^{i\ov{i}} u_{\ov{p}} u_{q\ov{i}} T_{ip}^q.
\end{align*}
Thankfully, $\varphi '$ can be used to control the single derivatives of $u$ via \eqref{phibounds}. Combining this and \eqref{sumThetaii},
\begin{equation*}
(9) \leq  C + C \tr{\gtilde}{g} + \frac{1}{10} \varphi ' \sum_{i,p} \Theta^{i\ov{i}} \left( |u_{p\ov{i}}|^2 + |u_{pi}|^2 \right).
\end{equation*}
Together with (8) we have the upper bound
\begin{equation*}
(8)+(9) \leq  C + C \tr{\gtilde}{g} - \frac{9}{10} \varphi ' \sum_{i,p} \Theta^{i\ov{i}} \left( |u_{p\ov{i}}|^2 + |u_{pi}|^2 \right).
\end{equation*}

Combining the above estimates for the lines in \eqref{ddtLQ}, we have
\begin{align*}
0 & \leq  -(1+2c) \frac{H}{\eta_{1 \ov 1}} \\ \nonumber
& \ \ \ - \frac{c}{2} \sum_i \sum_{p \neq 1} \frac{\Theta^{i\ov{i}} |\eta_{p \ov{1} i} |^2}{(\eta_{1\ov{1}})^2} - \frac{c}{4} \sum_i \sum_{p \neq 1} \frac{\Theta^{i\ov{i}} |\eta_{1 \ov{p} i} |^2}{(\eta_{1\ov{1}})^2} + (1 + 2c) \sum_i \frac{ \Theta^{i\ov{i}} | \eta_{1 \ov{1} i } |^2 }{(\eta_{1\ov{1}})^2} \\ \nonumber
& \ \ \ - \nu '' \sum_i \Theta^{i\ov{i}} |u_i|^2 -  \varphi '' \sum_{i} \Theta^{i\ov{i}} \left \vert \sum_{p} u_p u_{\ov{p}i} + \sum_p u_{p i} u_{\ov{p}} \right \vert^2 \\ \nonumber
& \ \ \ +  C + C_0 \tr{\gtilde}{g} - \frac{9}{10} \varphi ' \sum_{i,p} \Theta^{i\ov{i}} \left( |u_{p\ov{i}}|^2 + |u_{pi}|^2 \right) \\ \nonumber
& \ \ \ - \frac{2(1+2c)}{\eta_{1\ov{1}}} \sum_{i \neq 1} \Theta^{i\ov{i}} \textrm{Re} \left( \overline{T_{1i}^1} \eta_{1\ov{1}i} \right) - (C_1 + 1) \tr{\gtilde}{\ghat}.
\end{align*}
This is the same inequality as part way through the second order estimate in \cite{TW5}. Since we are fixed at the point $(x_0,t_0)$, $\ghat$ is a fixed Hermitian metric. This lets us choose $C_1 > 0$ uniform and large such that
\begin{equation*}
(C_0 + 2) \tr{\gtilde}{g} \leq (C_1 + 1) \tr{\gtilde}{\ghat}. 
\end{equation*}
The remainder of the estimate goes through exactly as in \cite{TW5} and we will not reproduce it here. This gives the bound
\begin{equation*}
\eta_{1\ov{1}}(x_0,t_0) \leq CK
\end{equation*}
for uniform $C > 0$ which completes the proof as discussed above.
\end{proof}
 
\section{First order estimate}

Given the form of our second order estimate we require a first order estimate for $u$. For the proof we modify the argument of \cite{TW4} to apply in this parabolic setting.

\begin{lemma}\label{1est}
There exists a uniform $C > 0$ such that
\begin{equation}\label{nablaubound}
\sup_{M \times [0,S)} |\nabla u|^2_g \leq C.
\end{equation}
\end{lemma}
The proof of this lemma requires a bit of machinery which we will recall from \cite{TW4}. Let $\beta$ be the Euclidean K\"ahler form on $\C^n$ and $\Delta$ the Laplacian with respect to $\beta$. Let $\Omega \subset \C^n$ be a domain. We say that an upper semi-continuous function
\begin{equation*}
u : \Omega \to \R \cup \{-\infty \}
\end{equation*}
in $L_{loc}^1(\Omega)$ is $(n-1)$-PSH if
\begin{equation*}
P(u) := \frac{1}{n-1} \left( (\Delta u) \beta - \ddbar u \right) \geq 0
\end{equation*} 
as a $(1,1)$-current. A continuous $(n-1)$-PSH function $u$ is maximal if for any relatively compact open set $\Omega' \Subset \Omega$ and any continuous $(n-1)$-PSH function $v$ on a domain $\Omega ' \Subset \Omega '' \Subset \Omega$ and with $v \leq u$ on $\partial \Omega '$, then $v \leq u$ on $\Omega '$.  

We need the following Liouville-type theorem from \cite{TW4}. 
\begin{theorem}\label{liouville}(Tosatti-Weinkove)
If $u : \C^n \to \R$ is an $(n-1)$-PSH function in $\C^n$ which is Lipschitz continuous, maximal, and satisfies
\begin{equation*}
\sup_{\C^n} \left( |u| + |\nabla u| \right) < \infty
\end{equation*}
then $u$ is constant.
\end{theorem}
The proof of this result uses an idea of Dinew-Ko\l odziej \cite{DK}. With these definitions and the Liouville-type theorem, we now begin the proof of Lemma \ref{1est}.
\begin{proof}
Suppose for contradiction that \eqref{nablaubound} does not hold. Then there exists a sequence $(x_j,t_j) \in M \times [0,S)$ with $t_j \to S$ such that
\begin{equation*}
\lim_{j\to\infty}    |\nabla u(x_j, t_j)|^2_g  = \infty.
\end{equation*}
Without loss of generality we assume our $t_j$ are such that
\begin{equation*}
 \sup_{x \in M}  | \nabla u(x,t_j) |_g^2 = \sup_{M \times [0,t_j]}  | \nabla u |_g^2.
\end{equation*}
Additionally, we choose our $x_j$ to be a point at which $|\nabla u(\cdot,t_j)|_g$ attains its maximum. We define
\begin{equation*}
C_j := | \nabla u(x_j,t_j) |_g^2  = \sup_{M \times [0,t_j]}  | \nabla u|_g^2
\end{equation*}
which has the property $C_j \to \infty$ as $j \to \infty$.

With this setup, we are ready to apply the blow-up argument and the Liouville-type theorem from \cite{TW4} to obtain a contradiction. After passing to a subsequence, there exists an $x$ in $M$ such that $x_j \to x$ as $j \to \infty$. Fix holomorphic coordinates $(z^1,\ldots, z^n)$ centered at $x$ with $\omega(x) = \beta$ and identifying with the ball $B_2(0) \subset \C^n$. Also assume that $j$ is sufficiently large so that $x_j \in B_1(0)$. We define
\begin{align*}
u_j(x) & = u(x,t_j) \\
\Phi_j(z) & = C_j^{-1}z + x_j
\end{align*}
and
\begin{equation*}
\hat{u}_j(z) := \left(u_j \circ \Phi_j (z)\right) = u_j \left( C_j^{-1} z + x_j \right) \textrm{ for } z \in B_{C_j}(0).
\end{equation*}
Note that by construction $\hat{u}_j$ achieves its maximum at $z = 0$ and
\begin{equation}\label{ujnonconst}
|\nabla \hat{u}_j|_\beta(0) = C_j^{-1} | \nabla u(x_j) |_g = 1.
\end{equation}
We also have the uniform bounds
\begin{equation*}
\sup_{B_{C_j}(0)} |\hat{u}_j|_\beta \leq C, \ \ \ \sup_{B_{C_j}(0)} |\nabla \hat{u}_j|_\beta \leq 1.
\end{equation*}
Using Lemma \ref{2est} on $[0,t_j]$ (see \eqref{cpctinterval2est})
\begin{equation*}
\sup_{y \in M} |\ddbar u(y,t_j)|_g \leq C' \left( \sup_{M \times [0,t_j]} |\nabla u|_g^2 + 1 \right ) = C' C_j^2 + C'
\end{equation*}
which gives the estimate
\begin{equation*}
\sup_{B_{C_j}(0)} | \ddbar \hat{u}_j |_\beta \leq \frac{C}{C_j^2} \sup_{y \in M} | \ddbar u(y,t_j) |_g \leq C''.
\end{equation*}
For every compact $K \subset \C^n$, every $0 < \alpha < 1$, and every $p > 1$ there exists uniform $C > 0$ such that
\begin{equation*}
|| \hat{u}_j ||_{C^{1,\alpha}(K)} + ||\hat{u}_j||_{W^{2,p}(K)} \leq C
\end{equation*}
using the Sobolev embedding theorem. From this we have a function $u \in W^{2,p}_{loc}(\C^n)$ such that a subsequence $\hat{u}_j$ converges strongly in $C^{1,\alpha}_{loc}(\C^n)$ and weakly in $W^{2,p}_{loc}(\C^n)$ to $u$. Thus from the estimates for $\hat{u}_j$ we have the uniform bounds
\begin{equation*}
\sup_{\C^n} ( |u| + |\nabla u| ) \leq C
\end{equation*}
and from \eqref{ujnonconst} $u$ is nonconstant. Following the remainder of the argument for the elliptic case in \cite{TW4} shows that $u$ is maximal and is hence constant by the Liouville-type theorem, a contradiction. 
\end{proof}

\section{Higher order estimates and proof of the main theorem}

To finish the proof of the main theorem, it sufficed to prove the uniform higher order estimates
\begin{equation*}
||u||_{C^k(M,g)} \leq C_k
\end{equation*}
for $k = 0, 1, 2, \ldots$. With these estimates the flow converges smoothly as $t \to S$ to a metric $\omega_S$. We extend the flow to $[0,S]$ with $\omega_t|_{t=S} = \omega_S$ allowing us to begin the flow once more. This contradicts the fact that $S$ is maximal so we must have $S = T$ since the flow cannot exist beyond $T$. We now prove the higher order estimates.

Summarizing our current estimates for $u$, we have
\begin{equation*}
\sup_{M\times [0,S)} |u| + \sup_{M\times [0,S)}  |\nabla u|_g + \sup_{M\times [0,S)}  | \ddbar u |_g + \sup_{M\times [0,S)} |\dot{u}|_g \leq C
\end{equation*}
for a uniform $C > 0$. Note that from the volume form bound in Lemma \ref{0est} and the trace bound in Lemma \ref{2est} we have that $\gtilde$ is uniformly equivalent to g:
\begin{equation}\label{unifequiv}
\frac{1}{C} g \leq \gtilde \leq C g.
\end{equation}
Using standard parabolic theory, the higher order estimates follow from a uniform parabolic $C^{2+\alpha}(M,g)$ bound for $u$ for some $\alpha > 0$. This can be done via the parabolic Evans-Krylov method as in \cite{G} with some modification (also see \cite{GT,L}).

Let $B_R$ be a small ball in $\mathbb{C}^n$ of radius $R > 0$ centered at the origin. Let $\varepsilon > 0$ and fix $t_0 \in [\varepsilon, T)$. We work in the parabolic cylinder
\begin{equation*}
Q(R,t_0) = \left \{ (x,t) \in B_R \times [0,S) \ | \  t_0 - R^2 \leq t \leq t_0 \right \}.
\end{equation*}    
Let $\{ \gamma_i \}$ be a basis for $\mathbb{C}^n$. For the $C^{2+\alpha}(M,g)$ estimate it suffices to prove the bound
\begin{equation*}
\sum_{i=1}^n \textrm{osc}_{Q(R,t_0)}(u_{\gamma_i \ov{\gamma}_i}) + \textrm{osc}_{Q(R,t_0)}(\dot{u}) \leq C R^\delta
\end{equation*}
for any $t_0 \in [\varepsilon,S)$, for some uniform $C > 0$, some $R > 0$ sufficiently small, and some $\delta > 0$.

We first rewrite the flow \eqref{ddtu} as
\begin{equation}\label{logdet}
-\ddt u + \log \det \gtilde = \tilde{F}
\end{equation}
where $\tilde{F} = \psi/S + \log \Omega.$ Let $\gamma$ be an arbitrary unit vector in $\mathbb{C}^n$. We differentiate the flow covariantly and commute derivatives as in \eqref{dotu11} and \eqref{dotu112} to  obtain
\begin{equation*}
- \ddt u_{\gamma \ov{\gamma}} + \Theta^{i\ov{j}} u_{i \ov{j} \gamma \ov{\gamma}} \geq G + \frac{H}{\eta_{1\ov 1}} - C \sum_{p,q} |u_{p\ov{q}\gamma}|
\end{equation*}
where $G$ is bounded function (using our existing uniform estimates) and
\begin{equation*}
H = \frac{\gtilde^{i\ov{q}} \gtilde^{p\ov{j}} \left(g_{p\ov{q}} g^{r\ov{s}} u_{r \ov{s} \ov{\gamma}} - u_{p \ov{q} \ov{\gamma}} + \hat{h}_{p\ov{q}\ov{\gamma}} \right) \left(g_{i\ov{j}} g^{a\ov{b}} u_{a\ov{b}\gamma} - u_{i \ov{j} \gamma} + \hat{h}_{i\ov{j}\gamma}\right)}{(n-1)^2}
\end{equation*}
as in \eqref{Hdef}. Converting the covariant derivatives to partial derivatives,
\begin{equation*}
- \ddt u_{\gamma \ov{\gamma}} + \Theta^{i\ov{j}}\partial_i  \partial_{\ov j} u_{ \gamma \ov{\gamma}} \geq G + H - C \sum_{p,q} |u_{p\ov{q}\gamma}|
\end{equation*}
for a larger $C > 0$. The latter two terms cancel because we have the estimate
\begin{align*}
\frac{H}{\eta_{1\ov 1}} & \geq \frac{1}{C'} \gtilde^{i\ov{q}} \gtilde^{p\ov{j}} \left(g_{p\ov{q}} g^{r\ov{s}} u_{r \ov{s} \ov{\gamma}} - u_{p \ov{q} \ov{\gamma}}  \right) \left(g_{i\ov{j}} g^{a\ov{b}} u_{a\ov{b}\gamma} - u_{i \ov{j} \gamma}\right) - C' \\
& \geq \frac{1}{C'} \left( (n-2) |g^{r\ov{s}} u_{r\ov{s}\gamma}|^2 +  \gtilde^{i\ov{q}} \gtilde^{p\ov{j}} u_{i\ov{j}\gamma} u_{p\ov{q}\ov{\gamma}} \right) - C' \\
& \geq C \sum_{p,q} |u_{p\ov{q}\gamma}|- C'
\end{align*}
for a uniform constant $C' > 0$, giving the bound
\begin{equation}\label{keyholder1}
- \ddt u_{\gamma \ov{\gamma}} + \Theta^{i\ov{j}} \partial_i \partial_{\ov j} u_{ \gamma \ov{\gamma}} \geq G.
\end{equation}
We also have
\begin{equation}\label{keyholder2}
-\ddt \dot{u} + \Theta^{i\ov{j}} \partial_i \partial_{\ov j} \dot{u} = \frac{(\tr{\gtilde}{g})(\tr{g}{\chi}) - \tr{\gtilde}{\chi}}{n-1} \leq C
\end{equation}
using \eqref{udotdot}, \eqref{ludot}, Lemma \ref{2est}, and Lemma \ref{1est} for a uniform $C > 0$. 

As in \cite{TW4,TW5} we define a metric $g'_{i\ov{j}} = g_{i\ov{j}}(x_0)$ on $B_R$. This fixed metric allows us to contract tensors that would otherwise be at different points in space and time. We will also use the tensor
\begin{equation*}
\hat{\Theta}^{i\ov{j}} = \frac{1}{n-1}\left( (\tr{\gtilde}{g'}) g'^{i\ov{j}} - \gtilde^{i\ov{j}} \right)
\end{equation*}
and the operator
\begin{equation*}
\Delta' = g'^{i\ov{j}} \partial_i \partial_{\ov{j}}.
\end{equation*}
By the mean value inequality, for all $x$ in $B_R$,
\begin{equation}\label{gprimebound}
|g'_{i\ov{j}}(x) - g_{i\ov{j}}(x) | \leq CR.
\end{equation}
We let $\Phi$ be an operator on a matrix $A$ given by $\Phi(A) = \log \det A$. Since $\Phi$ is concave, for all $(x,t), (y,s) \in B_R \times [0,S)$
\begin{equation}\label{concave}
\sum_{i,j} \frac{\partial \Phi}{\partial a_{i\ov{j}}} \left( \gtilde(y,s) \right) \left( \gtilde_{i\ov{j}}(x,t) - \gtilde_{i\ov{j}}(y,s) \right) \geq \Phi(\gtilde(x,t)) - \Phi(\gtilde(y,s)).
\end{equation}
Using \eqref{logdet}, equation \eqref{concave} becomes
\begin{equation}\label{crbound1}
\dot{u}(x,t) - \dot{u}(y,s) + \sum_{i,j} \gtilde^{i\ov{j}}(y,s) \left(  \gtilde_{i\ov{j}}(y,s) - \gtilde_{i\ov{j}}(x,t) \right) \leq C R
\end{equation}
after applying the mean value inequality to $\tilde{F}$. We need to further bound the last term on the left hand side. Computing from the definition of $\gtilde$
\begin{align}\label{metricsholder}
 \sum_{i,j} & \gtilde^{i\ov{j}}(y,s)  \left(  \gtilde_{i\ov{j}}(y,s)  - \gtilde_{i\ov{j}}(x,t) \right)  =  \sum_{i,j} \gtilde^{i\ov{j}}(y,s) \left( \ghat_{i\ov{j}}(y,s) - \ghat_{i\ov{j}}(x,t) \right) \\ \nonumber
 &  + \frac{1}{n-1}  \sum_{i,j} \gtilde^{i\ov{j}}(y,s) \left( ( (\Delta u) g_{i\ov{j}} - u_{i\ov{j}} )(y,s) -  ( (\Delta u) g_{i\ov{j}} - u_{i\ov{j}} )(x,t) \right).
\end{align}
The mean value inequality in $Q(R,t_0)$ along with the uniform bounds for $\gtilde$ and $\ghat$ gives
\begin{equation}\label{mvigtilde}
\left\vert \sum_{i,j} \gtilde^{i\ov j}(y,s) \left( \ghat_{i \ov j}(y,s) - \ghat_{i\ov j}(x,t) \right) \right\vert \leq CR.
\end{equation}
Then with \eqref{metricsholder}, \eqref{mvigtilde}, and the uniform bounds for $u_{i\ov j}$ and $\gtilde_{i\ov j}$ equation \eqref{crbound1} becomes
\begin{align}\label{crbound2}
&  \frac{1}{n-1}  \sum_{i,j} \gtilde^{i\ov{j}}(y,s) \left( ( (\Delta' u) g'_{i\ov{j}} - u_{i\ov{j}} )(y,s) -  ( (\Delta u') g'_{i\ov{j}} - u_{i\ov{j}} )(x,t) \right) \\ \nonumber
& \ \ \  + \dot{u}(x,t) - \dot{u}(y,s) \leq CR.
\end{align}
Here is where we use the fixed metric $g'$. Since
\begin{equation*}
\sum_{i,j} \hat{\Theta}^{i\ov{j}}(y,s) u_{i\ov{j}}(z,r) = \sum_{i\ov{j}} \gtilde^{i\ov{j}}(y,s) \left( (\Delta' u)g'_{i\ov{j}} - u_{i\ov{j}} \right)(z,r),
\end{equation*}
for any $(z,r) \in B_R \times [0,S)$ we have the estimate
\begin{equation}\label{keyholder3}
\dot{u}(x,t) - \dot{u}(y,s) + \sum_{i,j} \hat{\Theta}^{i\ov{j}}(y,s) \left( u_{i\ov{j}}(y,s) - u_{i\ov{j}}(x,t) \right) \leq CR.
\end{equation}
Following \cite{G} (or \cite{GT,L}) we find finitely many unit vectors $\gamma_1, \ldots, \gamma_n$ in $\C^n$ and real valued functions $\beta_\nu$ on $B_R \times [0,S)$ with
\begin{equation*}
0 < C^{-1} < \beta_\nu < C
\end{equation*}
for $\nu = 1, \ldots, N$ such that
\begin{equation*}
\hat{\Theta}^{i\bar{j}}(y,s) = \sum_{\nu = 1}^N \beta_\nu (y,s) \left(\gamma_\nu\right)^i \ov{ \left( \gamma_\nu \right)^ j }.
\end{equation*}
For $\nu = 1, \ldots, N$ define
\begin{equation*}
w_\nu = u_{\gamma_\nu \ov{\gamma_\nu}}
\end{equation*}
and for $\nu = 0$,
\begin{equation*}
w_0 = -\dot{u}, \textrm{ and } \beta_0 = 1.
\end{equation*}
From \eqref{keyholder3},
\begin{equation}\label{keyholder4}
\sum_{\nu = 0}^N \beta_\nu (y,s) \left( w_\nu (y,s) - w_\nu (x,t) \right) \leq CR
\end{equation}
and for all $\nu = 0, 1, \ldots, N$,
\begin{equation}\label{keyholder5}
-\ddt w_v + \Theta^{i\ov{j}} \partial_i \partial_{\ov j} w_\nu \geq G 
\end{equation}
where $G$ is a uniformly bounded function using \eqref{keyholder1} and \eqref{keyholder2}. With the key estimates \eqref{keyholder4} and \eqref{keyholder5} we can complete the $C^{2+\alpha}(M,g)$ estimate exactly as in \cite{G} for the parabolic complex Monge-Amp\'ere equation. This finishes the proof of the main theorem.

\section{Acknowledgments}

The author thanks Valentino Tosatti for several helpful discussions and suggestions and especially thanks Ben Weinkove for his encouragement, support, and advice.

\bigskip
\noindent

\end{document}